\documentclass[12pt]{article}

\usepackage{amssymb,amsmath,amsthm}
\usepackage{graphicx}
\usepackage[all]{xy}
\usepackage{mathtools}
\usepackage{color}
\usepackage{hyperref}
\usepackage[nottoc]{tocbibind}

\numberwithin{equation}{section}

\theoremstyle{theorem} 
\newtheorem{theorem}{Theorem}[section]
\newtheorem{cor}[theorem]{Corollary}
\newtheorem{lem}[theorem]{Lemma}
\newtheorem{prop}[theorem]{Proposition}
\newtheorem{question}[theorem]{Question}

\theoremstyle{definition}
\newtheorem{defn}[theorem]{Definition}
\theoremstyle{remark} 
\newtheorem{rem}[theorem]{Remark}

\newcommand{\bb}[1]{\mathbb{#1}}
\renewcommand{\rm}[1]{\mathrm{#1}}
\renewcommand{\cal}[1]{\mathcal{#1}}
\newcommand{\mM}{\mathcal{M}}
\newcommand{\nN}{\mathcal{N}}
\newcommand{\pP}{\mathcal{P}}
\newcommand{\cC}{\mathcal{C}}
\newcommand{\dD}{\mathcal{D}}
\newcommand{\qQ}{\mathcal{Q}}
\newcommand{\tT}{\mathcal{T}}
\newcommand{\Hom}{\mathrm{Hom}}
\newcommand{\uHom}{\underline{\mathrm{Hom}}}
\newcommand{\Mod}{\mathrm{Mod}}
\newcommand{\coMod}{\mathrm{coMod}}
\newcommand{\bp}{\otimes_k}

\title{Tannakian formalism for fiber functors over tensor categories}
\author {Amir Jafari \\ Mostafa Einollahzadeh}

\begin{document}
\maketitle
\begin{abstract}
In this paper we generalize Tannakian formalism to fiber functors over general tensor categories. We will show that (under some technical conditions) if the fiber functor has a section, then the source category is equivalent to the category of comodules over a Hopf algebra in the target category. We will also give a description of this Hopf algebra using the notion of framed objects.
\end{abstract}

\tableofcontents 
\section{Introduction}
The theory of Tannakian categories (developed by Saavedra \cite{S}, Deligne \cite{DM}, \cite{D} and others) gives a description of tensor categories $\cC$ (over a field $k$), equipped with a fiber functor $F$ (i.e. an exact faithful $k$-linear tensor functor) to the category $\mathrm{Vect}_k$ of finite dimensional vector spaces over $k$. These categories are called ``neutral Tannakian categories" and this theory gives a tensor equivalence between $\cC$ and the category $\coMod(A)$ of finite dimensional (over $k$) comodules over the Hopf algebra $A=\mathrm{End}(F)^*$. In this paper we consider a more general case of fiber functors $F:\cC\to \dD$ between two general tensor categories over $k$, and we will give simple criterions for having a Hopf algebra $A$ in $\dD$, such that $F$ becomes equivalent to the forgetful functor from the category $\coMod(A)$ of comodules over $A$ in $\dD$  to $\dD$.

For a Hopf algebra $A$ in $\dD$, the category $\coMod(A)$ contains a full subcategory of trivial comodules which is equivalent under the forgetful functor to $\dD$. This gives a trivial condition for the existence of $A$: that is the existence of a ``section" for $F$, i.e. a tensor functor $s:\dD\to \cC$ with a tensor equivalence $F\circ s \simeq id_\dD$. In fact we will see that under some other technical conditions, the existence of a section is also a sufficient condition for the existence of $A$. But an important point is that $A$ is not simply the dual of $\underline{\rm{End}}(F)$ (inner endomorphism object of $F$ in $\dD$), but  it is a quotient of this dual object  which depends on the given section $s$. For doing this, we present two approaches: 

In the first approach, the fiber functor takes values instead of $\dD$, in the category of finitely generated $S$-modules in a semisimple tensor category $\pP$, where $S$ is a commutative ring in $\pP$, which replaces $Vect_k$ in the classical theory. The main idea comes from a simple observation that in the above context, $s$ defines an action of $\dD$ on $\cC$ via $(D,C)\mapsto s(D)\otimes C$  for $D\in \dD, C\in \cC$. This action defines a $\dD$-module category structure on $\cC$ and $F$ becomes a $\dD$-module functor from $\cC$ to $\dD$ (or $\Mod(S)$ for a commutative algebra $S$ in $\dD$). Focusing on this structure, we start in section \ref{mod} by some fundamental definitions and results on module categories over a semisimple rigid monoidal category $\pP$. In the next section we consider $\pP$-module functors $\omega$ from a $\pP$-module category $\mM$ to the category of dualizable modules over an algebra $S$ in $\pP$. Here we associate a coalgebroid $L(\omega)$ over $S$ to such functors ({which is the  dual of the $\pP$-module inner endomorphism object of $\omega$}) and in  subsection \ref{exact-faithful-functors} we show that if $\omega$ is exact and faithful, then (under some extra conditions) $\mM$ is equivalent to the category of finitely generated  $S$-modules with a comodule structure  over $L(\omega)$. Now if we add a symmetric monoidal structure to $\mM$ and assume that $\omega$ is also symmetric monoidal (with respect to tensor product of modules over a commutative algebra $S$ in a symmetric monoidal category $\pP$), this induces  a Hopf algebroid structure over $S$ on $L(\omega)$ and finally in Theorem \ref{main} we show that $\mM$ is equivalent (as a monoidal and $\pP$-module category) to $\coMod(L(\omega))$.

In our second approach (section \ref{fund-grp}), we consider general tensor categories $\dD$ and fiber functors $F:\cC\to \dD$. Here we use Deligne's results in \cite{D} on the fundamental groups of tensor categories, and finally in the Corollary \ref{fiber-tensor}, we show that (under some technical conditions) if there is a section $s:\dD\to \cC$, then $\cC$ is equivalent to the category $\rm{Rep}_\dD(G)$ of representations in $\dD$ of an ``affine group scheme" $G$ over $\dD$ (or a $\dD$-group) and $F$ corresponds under this equivalence to the forgetful functor $\rm{Rep}_\dD(G)\to \dD$. $G$ itself can be described as ``$k$-tensor automorphisms" of $F$  which act as identity on $F(s(X))$ for every $X\in \dD$. This shows the dependence of $G$ on $s$, and in fact there are simple examples (Remark \ref{grp-sec}) where two different sections for $F$ give two nonisomorphic $\dD$-groups. The Hopf algebra which corresponds to $G$ is the same as the Hopf algebra associated to $F$ as a $\dD$-module functor in previous sections.

Fiber functors to tensor categories (besides the category of vector spaces) arise naturally in several places in algebraic geometry. Here we give two typical examples:
\begin{itemize}
\item Let $X$ be an smooth variety over $\mathbb{C}$ and $\cC$ be the category of admissible (or unipotent, etc.) variations of mixed (or pure) Hodge structures on $X$ and $F$ be the fiber functor to the category $\dD$ of variations of the same kind over the point (i.e. mixed or pure Hodge structures), which is defined on  $\cC$ by the restriction functor to a fixed point $x\in X$. Then constant variations define a natural section for $F$ and so we have a tensor equivalence between $\cC$ and the category of representations of a  $\dD$-group in $\dD$. 

In the second section of \cite{A}, D. Arapura develops a general theory of an  ``enriched local systems" $E$ over smooth complex varieties generalizing above examples. He associates to every smooth variety $X$ and a point $x\in X$, a ``$E$-fundamental group" $\pi_1^E(X,x)$ which is a $E(pt)$-group. Then he shows that there is an equivalence between $E(X)$ (enriched local systems over $X$) and $\rm{Rep}_{E(pt)}(\pi_1^E(X,x))$. Arapura's construction of the fundamental group coincides with our construction of the associated $\dD$-group in Corollary \ref{fiber-tensor}.
\item
Let $\mM$ be some nice ``mixed category" (such as mixed Hodge structures, (hypothetical) mixed motives, mixed Tate motives or (hypothetical) mixed elliptic motives, etc.) and $F$ be the fiber functor defined by the sum of graded quotients to the semisimple tensor subcategory $\pP$ of  pure objects in $\mM$. In this case, the containment of $\pP$ in $\mM$ is an evident section for $F$ and our results give a description of $\mM$ as the category of $\pP$-objects with a comodule structure over a Hopf algebra in $\pP$. 

In section 3 of \cite{G}, A. Goncharov considers a general ``mixed category" $\mM_\pP$ over a semisimple tensor category $\pP$ (this means that  $\mM_\pP$ is a tensor category containing $\pP$, with a  ``nice" filtration $W$ on every object such that the graded quotients belong to $\pP$). Then he defines the functor:
$$\Psi: \mM_\pP\to \pP, \quad M\mapsto \oplus_{n\in \mathbb{Z}} gr_n^W M,$$
and $H(\mM_\pP):= \underline{\rm{End}}(\Psi)$. He claims that $\mM_\pP$ is equivalent to category of $H(\mM_\pP)$-modules in $\pP$ and then he gives a description of the dual of $H(\mM_\pP)$ by using  ``framed objects".\footnote{The notion of framed objects  appeared first in \cite{BGSV}.} (theorem 3.3 in \cite{G}, the object defined by framed objects is  denoted by $H'(\mM_\pP)$) As we mentioned above, $H(\mM_\pP)$ is not the right Hopf algebra to consider, (it can be easily seen in the trivial case $\mM_\pP=\pP$ which $H(\mM_\pP)$ is a nontrivial Hopf algebra in general) but a certain subalgebra of it is the appropriate Hopf algebra, which we call it the dual of the ``associated coalgebra". 

We will generalize   Goncharov's notion of framed objects in section \ref{fiber-sec} to general module functors, and in Proposition \ref{framed-coalg}, we will show that this gives a similar formula for the description of the associated coalgebras (or coalgebroids in general). So the Hopf algebra $H'(\mM_\pP)$ is the right object to consider and Goncharov's mistake in proving theorem 3.3 is essentially  only to consider that ``$\mM_\pP$ is equivalent to the category of $\underline{\rm{End}}(\Psi)$-modules" as a consequence of Tannakian formalism for fiber functors over semisimple categories.
\end{itemize}

It is also worth to mention that there are some connections between the results in this paper and Day's work in \cite{Day} on the reconstruction of a coalgebra or a Hopf algebra from its category of comdules and the fiber functor which forgets the comodule structure. But our main problem as discussed above is different, because our main interest is to find sufficient conditions for a category (with a fiber functor and extra structures) to be equivalent to a category of comdules for some coalgebra (or Hopf algebra) and finding formulas for the corresponding coalgebra in such cases. Another difference is that in working with the general case of fiber functors over a general monoidal category, we used the notions of module categories and functors rather than the notions of enriched categories and functors. This mostly due to the fact that (at least in the cases we studied) the module conditions are more restrictive than the enriched conditions, which we think are necessary for our purposes. For a more detailed discussion, see Remark \ref{module-enriched}.

\section{Module categories}\label{mod}

Throughout this paper $k$ is an arbitrary field, and all  categories and functors are $k$-linear additive and all categories are essentially small. For two such categories $\cal{C}$ and $\cal{D}$, $\cal{C}\otimes \cal{D}$ denotes the category whose objects are pairs of objects in $\cal{C}$ and $\cal{D}$ and the set of  morphisms between $(X,Y)$ and $(X',Y')$ is $\Hom_\cal{C}(X,X')\otimes_k \Hom_\cal{D}(Y,Y')$. 

Gabber's theorem on the characterization of locally finite and finitely generated $k$-linear categories as a category of modules over an algebra (cf. \cite{D}, 2.14), can be viewed as a result about module categories over the category $\rm{Vect}_k$ of finite dimensional vector spaces over $k$. In this section we will try to give a generalized version of this theorem for module categories over a general $k$-linear semisimple abelian rigid monoidal category. 
In the following basic definitions, we follow two main references \cite{O} and \cite{E} with minor changes for restricting to the case of $k$-linear categories.

\begin{defn} A {\em monoidal category} consists of
the following data: a category $\cal{C}$, a functor $\otimes : \cal{C} \otimes \cal{C} \to \cal{C}$,
functorial isomorphisms $a_{X,Y,Z}: (X\otimes Y)\otimes Z\to X\otimes
(Y\otimes Z)$, a unit object ${\bf 1} \in \cal{C}$, functorial isomorphisms $r_X: 
X\otimes {\bf 1} \to X$ and $l_X: {\bf 1} \otimes X\to X$ subject to the following 
axioms:

1) Pentagon axiom: the diagram
$$
\xymatrix{&((X\otimes Y)\otimes Z)\otimes T \ar[dl]_{a_{X,Y,Z}\otimes id} 
\ar[dr]^{a_{X\otimes Y,Z,T}}&\\
(X\otimes (Y\otimes Z))\otimes T \ar[d]^{a_{X,Y\otimes Z,T}}&&(X\otimes Y)
\otimes (Z\otimes T) \ar[d]_{a_{X,Y,Z\otimes T}}\\X\otimes ((Y\otimes Z)
\otimes T) \ar[rr]^{id\otimes a_{Y,Z,T}}&&X\otimes (Y\otimes (Z\otimes T))}
$$
commutes.

2) Triangle axiom: the diagram
$$
\xymatrix{(X\otimes {\bf 1} )\otimes Y\ar[rr]^{a_{X,{\bf 1} ,Y}} \ar[dr]^{r_X\otimes 
id}&&X\otimes ({\bf 1} \otimes Y)
\ar[dl]_{id\otimes l_Y}\\ &X\otimes Y&}
$$
commutes.
\end{defn}

\begin{defn} (i) Let $\cal{C}$ be a monoidal category
and $X$ be an object in $\cal{C}$. A {\em right dual} to $X$ is an object $X^*$
with two morphisms
$$
e_X: X^*\otimes X\to {\bf 1},\; i_X: {\bf 1} \to X\otimes X^*
$$
such that the compositions
$$ X\stackrel{i_X\otimes id}{\longrightarrow}X\otimes X^*\otimes X
\stackrel{id \otimes e_X}{\longrightarrow}X$$
$$ X^*\stackrel{id\otimes i_X}{\longrightarrow}X^*\otimes X\otimes X^*
\stackrel{e_X \otimes id}{\longrightarrow}X^*$$
are equal to the identity morphisms.

(ii) A {\em left dual} to $X$ is an object $^*X$ with two morphisms
$$
e_X': X\otimes {}^*X\to {\bf 1},\; i_X': {\bf 1} \to {}^*X\otimes X
$$
such that the compositions
$$ X\stackrel{id \otimes i_X'}{\longrightarrow}X\otimes {}^*X\otimes X
\stackrel{e_X'\otimes id}{\longrightarrow}X$$
$$ ^*X\stackrel{ i_X'\otimes id}{\longrightarrow}{}^*X\otimes X\otimes {}^*X
\stackrel{id \otimes e_X' }{\longrightarrow}{}^*X$$
are equal to the identity morphisms.

(iii) A monoidal category $\cal{C}$ is called {\em rigid} if every object in 
$\cal{C}$ has right and left duals.
\end{defn}

\begin{defn} A {\em module category} over a monoidal category $\cal{C}$ is
a category $\cal{M}$ together with a functor $\otimes: \cal{C} \otimes \cal{M} \to 
\cal{M}$ and functorial associativity and unit isomorphisms $m_{X,Y,M}:
(X\otimes Y)\otimes M
\to X\otimes (Y\otimes M),\; l_M: {\bf 1} \otimes M\to M$ for any $X, Y\in \cal{C},\;
M\in \cal{M}$ such that the diagrams
$$
\xymatrix{&((X\otimes Y)\otimes Z)\otimes M \ar[dl]_{a_{X,Y,Z}\otimes id} 
\ar[dr]^{m_{X\otimes Y,Z,M}}&\\
(X\otimes (Y\otimes Z))\otimes M \ar[d]^{m_{X,Y\otimes Z,M}}&&(X\otimes Y)
\otimes (Z\otimes M) \ar[d]_{m_{X,Y,Z\otimes M}}\\X\otimes ((Y\otimes Z)
\otimes M) \ar[rr]^{id\otimes m_{Y,Z,M}}&&X\otimes (Y\otimes (Z\otimes M))}
$$
and
$$
\xymatrix{(X\otimes {\bf 1} )\otimes M\ar[rr]^{m_{X,{\bf 1} ,Y}} \ar[dr]^{r_X\otimes 
id}&&X\otimes ({\bf 1} \otimes M)
\ar[dl]_{id\otimes l_M}\\ &X\otimes M&}
$$
commute.
\end{defn}

\begin{rem}
(i) For any objects $M_1,M_2$ in a $\cal{C}$-module category $\mM$ and $X\in \cal{C}$ there are natural equivalences:
$$\Hom(X\otimes M_1,M_2) \simeq \Hom(M_1,{}^*X\otimes M_2),$$
$$\Hom(M_1, X\otimes M_2) \simeq \Hom(X^*\otimes M_1,M_2).$$

(ii) $X\otimes -$ is an exact functor $\mM\to \mM$.
\end{rem}

\begin{defn} (i) Let $\cal{M}_1$ and $\cal{M}_2$ be two module categories over
a monoidal category $\cal{C}$. A {\em module functor} from $\cal{M}_1$ to $\cal{M}_2$
is a functor $F: \cal{M}_1\to \cal{M}_2$ together with functorial isomorphisms
$c_{X,M}: F(X\otimes M)\to X\otimes F(M)$ for any $X\in \cal{C},\; M\in \cal{M}_1$
such that the diagrams
$$
\xymatrix{&F((X\otimes Y)\otimes M)\ar[dl]_{Fm_{X,Y,M}} \ar[dr]^{c_{X\otimes 
Y}}&\\ F(X\otimes (Y\otimes M))\ar[d]^{c_{X,Y\otimes M}}&&(X\otimes Y)
\otimes F(M)\ar[d]_{m_{X,Y,F(M)}}\\X\otimes F(Y\otimes M)\ar[rr]^{id\otimes 
c_{Y,M}}&&X\otimes (Y\otimes F(M))}
$$
and
$$
\xymatrix{F({\bf 1} \otimes M)\ar[rr]^{Fl_M} \ar[dr]^{c_{{\bf 1},M}}&&F(M)\\ 
&{\bf 1} \otimes F(M)\ar[ur]^{l_{F(M)}}&}
$$
are commutative.

(ii) Two module categories $\cal{M}_1$ and $\cal{M}_2$ over $\cal{C}$ are {\em equivalent}
if there exists a  module functor from $\cal{M}_1$ to $\cal{M}_2$ which is an
equivalence of categories.

(iii) A morphism of $\cal{C}$-module functors from $(F,c)$ to $(G,d):\mM_1\to\mM_2$ is a natural transformation $\nu$ between $F$ and $G$ such that the following diagram commutes for any $X \in \cal{C}$ and $M\in \mM_1$:
$$\xymatrix{
F(X\otimes M) \ar[r]^{c_{X,M}}\ar[d]_{\nu_{X\otimes M}} & X\otimes F(M)\ar[d]^{id\otimes \nu_M} \\ 
G(X\otimes M) \ar[r]^{d_{X,M}} & X\otimes G(M).}
$$
We denote the set of $\cal{C}$-module morphisms between $F$ and $G$ by $\Hom_\cal{C}(F,G)$.

\end{defn}

Throughout this paper $\cal{P}$ is a $k$-linear semisimple abelian rigid monoidal category with finite dimensional Hom-spaces.

\begin{defn}
Let
$\cal{M}$ be a module category over $\cal{P}$ and 
$M_1$ and $M_2$ be two objects of $\cal{M}$. Their \emph{internal hom} is an ind-object $\underline{Hom}(M_1,M_2)$ of $\cal{P}$ representing functor $X\mapsto \rm{Hom}(X\otimes M_1,M_2)$. 
\end{defn}

\begin{defn}
{An abelian} module category $\mM$ over $\pP$ is called \emph{locally finite over $\pP$} if 
\begin{itemize}
\item[(i)] The internal hom's of every two object in $\mM$ are objects in $\pP$, and
\item[(ii)] All objects have finite length.
\end{itemize}
\end{defn}

If $\mM$ be a locally finite module category over $\pP$, then for every two objects $X,Y\in \mM$, we have $\Hom_\mM(X,Y) \simeq \Hom_\pP({\bf 1}, \uHom(X,Y))$ and so $\Hom_\mM(X,Y)$ is finite dimensional over $k$. Here we used the term ``locally finite over $\pP$" to indicate that it is a generalization of the notion of locally finiteness for $k$-linear categories (cf. \cite{E}, 1.8) which are module categories over $\rm{Vect}_k$.

\begin{rem}\label{module-enriched}
If all of the internal hom's of a  $\pP$-module category $\mM$ belong to $\pP$, then  these internal hom's give a structure of an enriched category over $\pP$ on $\mM$. In this case the $\pP$-module category $\mM$ is usually called a \emph{tensored $\pP$-category}, (cf. \cite{B}, 6.5) which can also be defined starting from an enriched category $\mM$ over $\pP$ and assuming the existence of a tensor product $P\otimes M$ for every $P\in \pP$ and $M\in \mM$ together with natural isomorphisms:
$$\Hom_\mM(P\otimes M, -) \simeq \Hom_\pP(P,\uHom(M,-)).$$
The module categories we consider are usually tensored and hence enriched categories, but we focus on the module notion because our  main Theorems 
\ref{fiber} and \ref{main} are valid for module functors and not for every enriched functor. 

In fact every module functor between tensored categories is an enriched functor but the converse is not true. For a very simple example, consider $\mM=\rm{Vect}_k$ and $\pP=\rm{Vect}_k^\bullet$ (the category of finite direct sums of finite dimensional vector spaces graded by integers) with their standard monoidal structure. Then $\mM$ (considered as the category of graded vector spaces concentrated in grade 0) is a monoidal subcategory of $\pP$ and so is evidently enriched over $\pP$ and  this induces a module structure over $\pP$ for which the tensor product of $V^\bullet=\oplus_n V^n \in \pP$  and  $W\in \mM$ is given by $V^0\otimes W\in \mM$. Now it can be easily checked that the inclusion functor $\omega:\mM \hookrightarrow \pP$ is an enriched functor and does not satisfy the module condition. It is also interesting that $\mM,\pP$ and $\omega$ satisfy all of the conditions of Theorem \ref{fiber} except the module condition on $\omega$. Then one can show easily that $\mM$ is not equivalent to a category of comdules over any coalgebra over $\pP$, which shows the necessity of the module condition on $\omega$.

\end{rem}

\begin{defn}
{An abelian} module category $\cal{M}$ over $\cal{P}$ is \emph{finitely generated} over $\cal{P}$ if there exists an object $M$ in $\cal{M}$ such that every object of $\cal{M}$ is a subquotient of $P\otimes M$ for some object $P$ in $\cal{P}$. 
\end{defn}

\begin{defn}
A module category $\cal{M}$ over $\cal{P}$ is said to have \emph{Chevalley property} if for every two simple objects $P$ in $\cal{P}$ and $M$ in $\cal{M}$, $P\otimes M$ is semisimple. (This is a natural  extension of the definition of Chevalley property for tensor categories, cf. \cite{E}, 4.12)
\end{defn}

\begin{rem}\label{Chevalley}
Chevalley's theorem for the tensor product of semisimple representations of algebraic groups in characteristic 0,(cf. \cite{E}, Theorem 4.12.1) implies that if $\rm{char}(k)=0$, the tensor product of any two semisimple objects in a neutral Tannakian categery over $k$ is also semisimple. Now for every Tannakian category $\tT$ over $k$,\footnote{For definition of Tannakian categories, see section \ref{fund-grp}.} the category $\tT_{\bar{k}}$ constructed from $\tT$ by the extension of scalars to an algebraic closure $\bar{k}$ of $k$, is neutral Tannakian over $\bar{k}$.\footnote{For definition and basic properties of extension of scalars in Tannakian categories, see \cite{D2}, \S 4.} So $\tT_{\bar{k}}$ has the Chevalley property and by Corollary 4.11 in \cite{D2}, $\tT$ also has this property.

Now suppose $\mM$ be a Tannakian category over a field $k$ with characteristic 0, and $\pP$ be a tensor subcategory of the subcategory of semisimple objects in $\mM$. Then $\pP$ has an action defined by the tensor product of $\mM$, on $\mM$ and $\mM$ has the Chevalley property as a module category over $\pP$. 
\end{rem}
\begin{defn}
Let $R$ be an algebra in a monoidal category $\pP$ (it means that $R$ is an ind-object equipped with an associative product and a unit, when $R$ is an object of $\pP$ we say that $R$ is \emph{finite}) and $M$ be a right $R$-module in $\pP$. Then $M$ is called \emph{finitely generated} $R$-module if there is an epimorphism of right $R$-modules, $A\otimes R\to M$ for some $A\in \pP$. The category of finitely generated $R$-modules is denoted by $\rm{Mod}(R)$ and is a module category over $\pP$.
\end{defn}
\begin{rem}
i) A module over a finite algebra is finitely generated if and only if it is an object in $\pP$.

ii) The category of all right $R$-modules over an algebra $R$ is equivalent to the ind-category of the category of  finitely generated $R$-modules.
\end{rem}

The following theorem is a generalization of Gabber's theorem (\cite{D}, 2.14). An  special case of this theorem is proved in Ostrik's paper\cite{O} when $\cal{M}$ is semisimple and $\cal{P}$ has finitely many equivalence classes of simple objects.

\begin{theorem}\label{gGabber}
Let $\cal{M}$ be {an abelian} locally finite module category over $\cal{P}$. Suppose further that $\cal{M}$ is finitely generated over $\cal{P}$ and has the Chevalley property. Then $\cal{M}$ is equivalent to the category of finitely generated right modules over a finite algebra $R$ in $\cal{P}$.
\end{theorem}

The main idea of the proof is to show the existence of a projective generator in $\cal{M}$. We recall some definitions from \cite{D}. An epimorphism $M\to N$ in an abelian category is called \emph{essential} 
if there is no proper subobject $M'\subset M$ such that the composite $M'\to M\to N$ is an epimorphism. We also denote the number of composition factors of $M$ which are isomorphic to a simple object $S$ by $\ell_S(Y)$.

\begin{lem}\cite{D}\label{lem1}
Suppose $\cal{M}$ is a $k$-linear {abelian} category with finite dimensional Hom-spaces and all objects have finite length. If $Q\to S$ is an essential surjection to a simple object $S$ in $\cal{M}$ then for every object $Y$,
$$\rm{dim}_k \rm{Hom}(Q,Y) \leq \ell_S(Y)\cdot \rm{dim}_k \rm{End}(S).$$
The equality holds for every $Y$ if and only if it holds for a set of generators (over $k$) if and only if $Q$ is projective. Moreover if $\cal{M}$ is finitely generated (over $k$), for every simple object $S$ there is an essential surjection $Q\to S$ from a projective object $Q$.
\end{lem}

\begin{proof}[Proof of $\,$\ref{gGabber}]
Let $S$ and $T$ be two simple objects in $\cal{M}$. Then for every simple object $A\in\cal{P}$, we have
$$\rm{Hom}(A\otimes S, T) = \rm{Hom}(A, \underline{\rm{Hom}}(S,T)).$$ 
Now $\underline{\rm{Hom}}(S,T)$ is in $\cal{P}$ and so there are finitely many equivalence classes of simple objects $A\in \cal{P}$ such that $\rm{Hom}(A\otimes S, T)$ is nonzero. By the Chevalley condition $A\otimes S$ is semisimple and so there are finitely many equivalence classes of simple objects $A\in \cal{P}$ such that $\ell_T(A\otimes S)>0.$ Decomposing an arbitrary object to simples shows that this result is true even when $S$ is not simple.

Now let $S_1,\dots,S_k$ be all of composition factors for a generator $X$ of $\cal{M}$ over $\cal{P}$, and $A_1,\dots,A_m$ be a complete set of representatives  for equivalence classes of simple objects $A$ in $\pP$ such that $S_i$ appears in the composition factors of $A\otimes X$. Now define $A= {\bf 1}\oplus A_1 \oplus\cdots\oplus A_m$ and let $\cal{N}$ be the full subcategory of subquotients of $(A\otimes X)^n$ in $\mM$. Then $\nN$ contains $X$ and satisfies the conditions of Lemma \ref{lem1}, so for each $1\leq i\leq k$ there exists some essential surjection $Q_i \to S_i$ in $\cal{N}$ such that every $Q_i$ is projective in $\cal{N}$. 

We claim that $Q_i$ is also projective in $\mM$. Again by lemma it is suffices to show that for every simple object $B\in \pP$,
$$\rm{dim}_k \rm{Hom}(Q_i,B\otimes X) = \ell_{S_i}(B\otimes X)\cdot \rm{dim}_k \rm{End}(S_i).$$
There are two cases: 1) If $B$ is isomorphic to one of $A_1,\dots,A_m$, then $B\otimes X$ lies in $\cal{N}$ and the equality holds because $Q_i$ is projective in $\cal{N}$. 2) If $B$ is not isomorphic to one of $A_i$'s then $\ell_{S_i}(B\otimes X)=0$ and so two sides are equal to zero. This justifies our claim and so $Q_i$'s are projective in $\mM$. 

Now define $Q= Q_1\oplus \cdots\oplus Q_k$ and $R= \underline{\rm{End}}(Q)$. Then $Q$ is a projective object in $\mM$ and the functor
$$Y \mapsto \underline{\rm{Hom}}(Q,Y)$$
defines a  module functor $F$ from $\mM$ to the category of  right $R$-modules. The locally finiteness of $\mM$ over $\pP$ implies that the internal Homs are in $\pP$ and so $F(Y)$ is finitely generated $R$-module for every $Y\in \mM$. We claim that $F:\mM\to \Mod(R)$ is an
equivalence of module categories over $\pP$.

First we show that $F$ is exact and faithful. Because $\pP$ is semisimple, for the exactness it suffices to show that for every $Z\in \pP$, $\rm{Hom}(Z,F(-))$ is exact. But we have
$$\rm{Hom}(Z,F(M)) = \rm{Hom}(Z,\underline{\rm{Hom}}(Q,M)) = \rm{Hom}(Q, {}^*Z\otimes M).$$
And $\rm{Hom}(Q, {}^*Z\otimes -) $ is exact by exactness of ${}^*Z\otimes -$ and projectivity of $Q$.

Now let $M$ be a nonzero object in $\mM$ and $S$ a simple subobject of $M$. Then $S$ is a subquotient of $C\otimes X$ for some  object $C\in \pP$ and so it is a subquotient of $D\otimes S_i$ for some simple object $D\in \pP$ and $1\leq i\leq k$. Therefore by the Chevalley property, $D\otimes S_i$ is semisimple and there is a surjection $D\otimes S_i\to S$. Combining this morphism with a surjection $D\otimes Q\to D\otimes S_i$ and the inclusion $S\hookrightarrow M$, yields a nonzero morphism $D\otimes Q\to M$. So 
$$\Hom(D,F(M)) = \Hom(D,\uHom(Q,M)) \simeq \Hom(D\otimes Q,M) \neq 0,$$
and $F(M)\neq 0$. So $F$ is an exact functor which sends every nonzero object to a nonzero object and therefore $F$ is faithful.

It remains to show that $F$ is full and essentially surjective. We have

\begin{multline*}
\Hom_{\mM}(C\otimes Q,M) \simeq \Hom_{\pP}(C,\uHom(Q,M)) \simeq \\
 \Hom_{R}(C\otimes R, F(M)) \simeq \Hom_{R}(F(C\otimes Q),F(M)).
\end{multline*}

Every object $N\in \mM$ is a quotient of $C\otimes Q$ for some $C\in \pP$ (it is because every simple object in the composition series of $N$ has this property and all of $C\otimes Q$'s are projective) and the kernel of this epimorphism is also a quotient of $D\otimes Q, D\in \pP$. So there is an exact sequence:
$$D\otimes Q\to C\otimes Q\to N\to 0.$$
Applying $\Hom(-,M)$ to this sequence and exactness of $F$, implies that
$$\Hom_\mM(N,M) \simeq \Hom_{R}(F(N),F(M)).$$

Now for proving the essential surjectivity, consider an arbitrary finitely generated module $Y$ over $R$. Then there is a natural exact sequence in ${\rm{Mod}(R)}$: ($Y, Y\otimes R \in \pP$)
$$Y\otimes R\otimes R\stackrel{f}{\to} Y\otimes R\to Y\to 0.$$
Now let $\tilde{f}:(Y\otimes R)\otimes Q\to Y\otimes Q$ be a morphism in $\mM$ which maps to $f$ under $F$. If $M$ be the cokernel of this morphism, then $F(M)\simeq Y$ and so F is essentially surjective.
\end{proof}

\begin{lem}\label{submodule}
With  assumptions of Theorem \ref{gGabber}, if $\nN$ is a full $\pP$-module subcategory of $\mM$, then there is a two sided ideal $I$ of $R$ such that the inclusion of $\nN$ in $\mM$ is equivalent to the inclusion of $\rm{Mod}(R/I)$ in $\rm{Mod}(R)$.
\end{lem}
\begin{proof}
The proof is completely similar to a corresponding result for $k$-linear categories (cf. \cite{D}, 2.18). The idea is that for a finite algebra $R$ and a full $\pP$-module subcategory $\cal{C}\subset \rm{Mod}(R)$, there is a maximal quotient $R/I$ of $R$ which lies in $\cal{C}$, $I$ is a two sided ideal and $R/I$ is a projective $\pP$-generator for $\cal{C}$. 
\end{proof}

\section{Module functors and coalgebroids}\label{fiber-sec}
\subsection{Associated coalgebroid of  a module functor}
\begin{defn}
Let $S$ be an algebra in $\pP$, then a right $S$-module $M$ is said to be \emph{dualizable} if there is a left $S$-module $M^*$ with a natural equivalence:
$$\Hom_S(A\otimes M, N) \simeq \Hom_{\pP}(A, N\otimes_S M^*), $$
for any $A\in \pP$ and any right $S$-module $N$.
\end{defn}

This definition is equivalent to have two morphisms:
$$e_M: M^*\otimes M \to S, \; i_M: {\bf 1}\to M\otimes_S M^*, $$
with the properties that the composition of the following morphisms are identity:
$$ M\stackrel{i_M\otimes id}{\longrightarrow}M \otimes_S M^* \otimes M \stackrel{id \otimes e_M}{\longrightarrow} M,$$
$$ M^* \stackrel{id \otimes i_M}{\longrightarrow} M^*\otimes M\otimes_S M^* \stackrel{e_M\otimes id}{\longrightarrow} M^*.$$

\begin{rem}
(i) Let $R,S$ be two algebras in $\pP$ and $M$ a $(R,S)$-bimodule which is a dualizable right $S$-module, then $M^*$ has a natural structure of a right $R$-module and for every right $R$-module $X$ and right $S$-module $Y$: 
$$\Hom_S(X\otimes_R M,Y) \simeq \Hom_R(X,Y\otimes_S M^*). $$

So there are two bimodule morphisms:
$$e_M^R: M^*\otimes_R M \to S, \; i^R_M: R\to M\otimes_S M^*,$$
satisfying properties similar to $(e_M,i_M)$.

(ii)Let $R,S$ be two algebras in $\pP$ and $M$ a dualizable right $S$-module, then for every right $R$-module $X$ and $(S,R)$-bimodule $Y$: 
$$\Hom_R(X, M\otimes_S Y) \simeq \Hom_{(S,R)}(M^*\otimes X, Y).$$ 
(iii) For a dualizable right module $M$ and $A\in \pP$, $(A\otimes M)^*$ is naturally equivalent to $M^*\otimes A^*$.
\end{rem}

The following lemma is an extension of Morita's lemma in the classical categories of modules:
\begin{lem}\label{right-exact}
Let $R$ and $S$ be two algebras in $\pP$ and $\omega:\rm{Mod}(R)\to \rm{Mod}(S)$ be a right exact $\pP$-module functor, then there exists a $(R,S)$-bimodule $M$ such that $\omega$ is $\pP$-module isomorphic to $-\otimes_R M$. Moreover for any two $(R,S)$-bimodules $M_1,M_2$ there is natural isomorphism:
$$\Hom_\pP(-\otimes_R M_1,-\otimes_R M_2) \simeq \Hom_{(R,S)}(M_1,M_2).$$
\end{lem}
\begin{proof}
Let $M = \omega(R)$. Then $M$ is a finitely generated right $S$-module. Now the multiplication morphism $p:R\otimes R\to R$ can be considered as a $R$-module morphism, which induces a $S$-morphism $\omega(p): R\otimes M\to M$. (We can extend $\omega$ to ind-objects of $\rm{Mod}(R)$ which are arbitrary right $R$-modules and the resulting functor is an $\rm{ind}(\pP)$-module functor also denoted by $\omega$.) By the associativity of $p$ and unit morphism of $R$, $\omega(p)$ satisfies the axioms for a left $R$-module which is compatible with the $S$-module structure of $M$. So $M$ is a $(R,S)$-bimodule.

Now let $X$ be an arbitrary right $R$-module with structure morphism $m:X\otimes R\to X$. Then there is a natural exact sequence:
$$X\otimes R\otimes R\xrightarrow{id\otimes p - m\otimes id} X\otimes R \stackrel{m}{\longrightarrow} X,$$
applying the exact module functor $\omega$ to this sequence implies that $\omega(X)$ is naturally isomorphic to $X\otimes_R M$.

For the second part, define $\omega_i = -\otimes_R M_i, \, i=1,2$. Then a $\pP$-morphism $\nu:\omega_1\to \omega_2$ gives a $S$-module morphism $\nu_R:M_1\to M_2$. By the previous part the $R$-module structure of $M_i$ can also be reconstructed from the module functor $\omega_i$ for $i=1,2$. This implies that $\nu_R$ is also a $(R,S)$-bimodule morphism. On the other hand a $(R,S)$-bimodule morphism $M_1\to M_2$ gives rise to a natural $\pP$-morphism $\omega_1\to \omega_2$ and this is inverse to the previous construction.
\end{proof}

\begin{lem}\label{proj-dual}
A right $S$-module is dualizable if and only if it is finitely generated and projective.
\end{lem}
\begin{proof}
First note that for every $P, A \in \pP$ and left $S$-module $N$, we have: 
\begin{multline*}
\Hom_S(A\otimes (P\otimes S) , N ) \simeq \Hom(A\otimes P, N) \simeq \Hom(A,N\otimes P^*) \\ \simeq \Hom(A, N\otimes_S(S\otimes P^*)).
\end{multline*}
So $P\otimes S$ has $S\otimes P^*$ as its dual and is dualizable. Also every finitely generated projective $S$-module is a direct summand of $P\otimes S$ for some $P\in \pP$, and has a dual which is a direct summand of $S\otimes P^*$.

Now let $M$ be a dualizable right $S$-module with dual $(M^*,e_M,i_M)$. Then if we consider the morphism $i_M$, compactness of $\bf 1$ implies the existence of a finitely generated submodule $N\subset M$ such that $i_M$ factors through $N\otimes_S M^*$. Then the first axiom for $(e_M,i_M)$ shows that there is a surjective morphism $N\to M$ and $M$ is also finitely generated.

On the other hand, by the definition of the dual module, we have:
$$\Hom_S(M,-) \simeq \Hom_S({\bf 1}\otimes M,-) \simeq \Hom({\bf 1}, -\otimes_S M^*). $$
But $\pP$ is semisimple and $\Hom(-,-)$ is exact in each variable. So $\Hom_S(M,-)$ is right exact and $M$ is projective.

\end{proof}

Every  $k$-linear additive category  $\mM$ has a canonical structure of a module category over $\rm{Vect}_k$. We denote this structure by $(V,M)\mapsto V\bp M$, (for $V\in \rm{Vect}_k, M\in \mM$) which is defined by the property:
$$\Hom_\mM(M_1,V\bp M_2) \simeq V\otimes_k \Hom_\mM(M_1,M_2).$$
Also if $V$ is infinite dimensional, $V\otimes_k M$ can be defined by the inductive limit of $V'\otimes_k M$ over finite dimensional subspaces $V'$ of $V$, and it is an object in $\rm{ind}(\mM)$.

The following lemma is a restatement of Yoneda's lemma:
\begin{lem}\label{Yoneda}
Let $\cal{C}$ and $\cal{D}$ be two $k$-linear categories and 
$F:\cal{C}\to \cal{D}, G:\cal{C}^{\rm{op}}\to \cal{D}$
be two $k$-linear additive functors. Then for every $A\in \rm{ind}(cal{C})$ there are two natural isomorphisms:
\begin{eqnarray*}
\int^{X\in \cal{C}} \Hom_\cal{C}(X,A)\bp F(X) &\simeq & F(A), \\
\int^{X\in \cal{C}} \Hom_\cal{C}(A,X)\bp G(X) &\simeq & G(A).
\end{eqnarray*} 
\end{lem}

(Our notation for  coend is taken from Mac Lane, \cite{M}, IX.6)
\begin{proof}
Let $B$ be an arbitrary object in $\cal{D}$. Giving a morphism 
$$\int^{\cal{C}} \Hom_\cal{C}(X,A)\bp F(X)\to B$$ 
is equivalent to give a natural morphism $\Hom_\cal{C}(-,A) \to \Hom_\cal{D}(F(-),B)$, which by Yoneda's lemma corresponds to an element in $\Hom_\cal{D}(F(A),B)$. So another use of Yoneda's lemma implies the first isomorphism. The second isomorphism is similar.
\end{proof}

\begin{defn}
Let $S$ be an algebra in $\pP$. A \emph{coalgebroid} $L$ over $S$ is a $(S,S)$-bimodule with two $(S,S)$-bimodule morphisms:
$$\Delta: L\to L\otimes_S L, \quad \epsilon: L\to S,$$
such that
$$(\Delta \otimes id)\circ \Delta = (id \otimes \Delta)\circ \Delta, \quad id = (\epsilon \otimes id)\circ \Delta = (id\otimes\epsilon)\circ \Delta. $$
A comodule over $L$ is a right module $M$ over $S$ with a morphism of $S$-modules:
$$\rho:M\to M\otimes_S L,$$
such that
$$(1\otimes \epsilon)\circ \rho = 1,\quad (\rho\otimes 1)\circ\rho = (1\otimes \Delta)\otimes \rho.$$
The category of comodules over $L$ which are finitely generated over $S$ is denoted by $\rm{coMod}(L)$.
\end{defn}


\begin{defn}
Let $\mM$ be a $\pP$-module category and $\omega$ a $\pP$-module functor from $\mM$ {to the category of dualizable right modules over a ring $S$ in $\pP$}. Then we define:
$$L(\omega) := \int^{M\in \mM} \Hom_S(\omega(M),S)\bp \omega(M). $$
We will show that $L(\omega)$ has a natural structure of a coalgebroid over $S$ and we call it  the \emph{ associated coalgebroid} of $\omega$.
\end{defn}

$L(\omega)$ has an evident right $S$-module structure. It has also a natural left $S$-module structure. For showing this, first note that by Lemma \ref{Yoneda}, $S$ can be written as follows:
$$S\simeq \int^{P\in \pP} \Hom(P,S)\bp P.$$
Now the action of a general term in this coend on the general term in the coend defining $L(\omega)$ is given by:
\begin{multline*}
(\Hom(P,S)\bp P)\otimes (\Hom_S(\omega(M),S)\bp \omega(M)) \simeq \\ \quad \left(\Hom(P,S)\otimes_k \Hom_S(\omega(M),S)\right) \bp \omega(P\otimes M) \longrightarrow \\
\qquad\qquad \Hom_S(P\otimes \omega(M), S\otimes S) \bp \omega(P\otimes M) 
\longrightarrow \\ \Hom_S(P\otimes \omega(M), S) \bp \omega(P\otimes M),
\end{multline*}
where the first arrow is simply the product of morphisms and the second is induced by the multiplication morphisms of $S$. it is easy to see that the above action induce a well defined action of coends which gives a left $S$-module structure on $L(\omega)$.

\begin{prop}\label{ass-coalg}
Let $\mM$ be a $\pP$ module category and {$\omega: \mM\to \rm{Mod}(S)$} be a $\pP$-module functor such that for any $M\in \mM$, $\omega(M)$ is dualizable. Then there is a universal $\pP$-modue morphism $\nu: \omega\to \omega\otimes_S L(\omega)$ which induces a natural isomorphism 
$$\Hom_{(S,S)}(L(\omega),L) \simeq \Hom_{\pP}(\omega,\omega\otimes_S L),$$
for any $(S,S)$-bimodule $L$.

The morphisms corresponding to $id: \omega\to \omega\otimes_S S$ and $(\nu\otimes id)\circ\nu : \omega\to \omega\otimes_S(L(\omega)\otimes_S L(\omega))$ give two morphisms
$$\epsilon: L(\omega)\to S, \; \Delta: L(\omega)\to L(\omega)\otimes_S L(\omega),$$
which induce a structure of a $S$-coalgebroid on $L(\omega)$. Furthermore for every $M$, $\nu_M$ is a comodule structure on $\omega(M)$ which induces a functor $\tilde{\omega}$ that makes the following diagram commutative: ($F$ is the forgetful functor)
$$
\xymatrix{
\mM \ar[r]^(.35){\tilde{\omega}}\ar[dr]_{\omega}& \rm{coMod}(L(\omega))\ar[d]^{F}\\ &\rm{Mod}(S)
}
$$
\end{prop}
\begin{proof}
Let $L$ be a $(S,S)$-bimodule and $\alpha: \omega\to \omega\otimes_S L$ be a $\pP$-module morphism. Then for every $M\in \mM$, we have a natural morphism of right $S$-modules: $\alpha_M: \omega(M)\to \omega(M)\otimes_S L$, which induces a $(S,S)$-bimodule morphism:
$$j_M: \omega(M)^*\otimes \omega(M) \to L.$$
On the other hand, $\alpha$ is $\pP$-module which means that for every $M\in \mM$ and $ P\in \pP$, the following diagram is commutative: (where $c$ is the $\pP$-module constraint of $\omega$)
$$\xymatrix{
\omega(P\otimes M)\ar[r]^(.45){\alpha_{P\otimes M}}\ar[d]_{c_{P,M}}& \omega(P\otimes M)\otimes_S L \ar[d]^{c_{P,M}\otimes id}\\
P\otimes \omega(M)\ar[r]^(.45){id \otimes \alpha_{M}}& P\otimes \omega(M)\otimes_S L
}
$$
which in turn is equivalent to the commutativity of the following diagram:
$$
\xymatrix{\omega(M)^*\otimes P^*\otimes P\otimes \omega(M) \ar[rr]^(.58){id\otimes e_P\otimes id }\ar[d]_{c_{P,M}^*\otimes c_{P,M}^{-1}} && \omega(M)^*\otimes\omega(M) \ar[d]^{j_M}\\
\omega(P\otimes M)^*\otimes \omega(P\otimes M) \ar[rr]^(.6){j_{P\otimes M}} && L}
$$
This puts another condition on the morphisms $j_M$ besides naturality. 

Now let us define $\lambda: \int^{ \mM} \omega(M)^*\otimes \omega(M)\to\Lambda$ to be the coequalizer of the following diagram:
\begin{equation}\label{eqn1}
 \int^{P\in\pP} \int^{M\in\mM} \omega(M)^*\otimes P^* \otimes P \otimes \omega(M) \xymatrix{\ar@<.4ex>[r]^(.5)\phi\ar@<-.4ex>[r]_(.5)\psi & }\int^{M\in \mM} \omega(M)^*\otimes \omega(M),   
\end{equation}
where the morphisms $\phi$ and $\psi$ are induced by the following morphisms respectively:
$$\omega(M)^*\otimes P^* \otimes P \otimes \omega(M) \xrightarrow{id\otimes e_P\otimes id} \omega(M)^*\otimes \omega(M),$$
$$\omega(M)^*\otimes P^* \otimes P \otimes \omega(M) \xrightarrow{c_{P,M}^*\otimes c_{P,M}^{-1}} \omega(P\otimes M)^*\otimes \omega(P\otimes M).$$
Then $\Lambda$ is a $(S,S)$-bimodule and the structure morphisms $i_M:\omega(M)^*\otimes \omega(M)\to \Lambda$ induce a $\pP$-module morphism $\mu:\omega\to \omega\otimes_S \Lambda$. Furthermore $\mu$ is universal and $\Lambda$ represents the functor $ \Hom_\pP(\omega,\omega\otimes_S -)$. Thus $\Lambda$ also satisfies all of the claimed properties of $L(\omega)$ and it only remains to show that $\Lambda \simeq L(\omega)$. 

For doing this we construct an inverse pair of $(S,S)$-bimodule morphisms $f:L(\omega)\to \Lambda, \, g:\Lambda\to L(\omega)$.

$f$ is defined by the composition of following morphisms:
\begin{eqnarray*}
L(\omega) & =& \int^\mM \Hom_S(\omega(M),S)\bp \omega(M) \\ & \simeq & \int^\mM \left(\Hom({\bf 1}, \omega(M)^*)\bp {\bf 1} \right) \otimes \omega(M)\\
&\longrightarrow & \int^\mM \omega(M)^*\otimes \omega(M) \xrightarrow{\lambda} \Lambda.
\end{eqnarray*}

Now consider the following sequence of morphisms: (The first isomorphism is by Lemma \ref{Yoneda} and the third is induced by $c_{A,M}$)
\begin{eqnarray*}
\omega(M)^*\otimes \omega(M) &\simeq &\int^{A\in\pP} \Hom(A, \omega(M)^*) \bp A \otimes \omega(M) \\
&\simeq & \int^{A\in \pP} \Hom_S(A\otimes \omega(M),S) \bp A \otimes \omega(M)\\
&\simeq& \int^{A\in \pP} \Hom_S(\omega(A\otimes M),S) \bp \omega(A \otimes M) \\
&\longrightarrow & \int^{X\in\mM} \Hom_S(\omega(X),S)\bp \omega(X) = L(\omega).
\end{eqnarray*}

If we denote the composition of above morphisms by $k_M$, then it is easy to see that $k_M$'s induces a morphism $\tilde{g}:\int^\mM \omega(M)^*\otimes \omega(M)\to L(\omega)$ and $\tilde{g}\circ \phi = \tilde{g}\circ \psi$, so $\tilde{g}$ induces a morphism $g:\Lambda\to L(\omega)$. 

it is also straightforward to check that $f$ and $g$ are bimodule morphisms, $f\circ g = id$ and $g\circ f = id$ which finishes the proof.

\end{proof}

By a complete similar proof, one can show that with the conditions of the above proposition,$\int^{M\in \mM} \Hom_S(S,\omega(M))\bp \omega(M)^*$ also satisfies the universal propery of $L(\omega)$ and this gives another formula for the associated coalgebroid of $\omega$:
\begin{equation}\label{second-form}
L(\omega) \simeq \int^{M\in \mM} \Hom_S(S,\omega(M))\bp \omega(M)^*.
\end{equation}

\begin{rem}
In this subsection the semisimplicity of $\pP$ is only used in Lemma \ref{proj-dual} for the characterization of dualizable objects.
\end{rem}

\subsection{Framed objects}
In this subsection we give a general definition of framed objects. 
Our motivation  is
the framed mixed Hodge structures considered first by Beilinson et al. in \cite{BGSV}.
\begin{defn}\label{framed}
Let $F: \cal{C}\to \cal{D}$ be a functor from a category $\cal{C}$ to a $k$-linear category $\cal{D}$. The functor of \emph{framed objects}, 
$$\cal{H}: \cal{D}^{\rm{op}}\times \cal{D}\to \rm{ind}(\rm{Vect}_k)$$
is defined by
$$\cal{H}(A,B)= \int^{X\in \cal{C}} \Hom_\cal{D}(A,F(X))\otimes_k \Hom_\cal{D}(F(X),B).$$

\end{defn}

i.e. informally one thinks of elements of $\cal{H}(A,B)$ as an object $X$ of $\cC$ together
with framings $\nu : A\to F(X)$ and $\hat{\nu} :  F(X)\to B$. The commutativity diagram
for defining the coend means that two such framed objects $(X,\nu,\hat{\nu})$ and $(X',\nu',\hat{\nu}')$ should be considered
equivalent if there is a morphism $X\to X'$ that respects the frames.

\begin{prop}\label{framed-coalg}
For a $\pP$-module category $\mM$ and a $\pP$-module functor $\omega:\mM\to \rm{Mod}(S)$, we have:
$$L(\omega) \simeq \int^{A\in\rm{Mod}(S)} \cal{H}(A,S)\bp A.$$
Also if  $S$ is semisimple (i.e. $\Mod(S)$ is a semisimple category),  then
$$L(\omega) \simeq \int^{A\in\rm{Mod}(S)} \cal{H}(S,A)\bp A^*.$$
\end{prop}
\begin{proof} For the first part:
\begin{align*}
\quad\quad &\hspace{-35pt}\int^{\rm{Mod}(S)} \cal{H}(A,S)\bp A  \\
&\simeq \int^{ \rm{Mod}(S)}\left.\int^{M\in\mM}\Hom_S(\omega(M),S)\otimes_k \Hom_S(A,\omega(M))\right. \bp A \\
&\simeq \int^\mM \Hom_S(\omega(M),S)\bp \left(\int^{\rm{Mod}(S)}\Hom_S(A,\omega(M))\bp A\right)\\
&\simeq \int^\mM \Hom_S(\omega(M),S)\bp \omega(M) \;=\; L(\omega).
\end{align*}
The second isomorphism is valid by the exactness of $\bp$ and the third by Lemma$\,$\ref{Yoneda}. 

For the second part, because  every object in $\Mod(S)$ is semisimple, every object is dualizable and the claimed isomorphism can obtained similarly using \eqref{second-form}.
\end{proof}

\begin{cor}
With the notations of the previous proposition, if $S$ is semisimple and $\cal{S}$ be a set of representatives for the equivalence classes of simple objects in $\Mod(S)$, then
$$L(\omega) \simeq  \bigoplus_{A\in \cal{S}} \cal{H}(A,S)\otimes_{\rm{End}(A)} A\simeq  \bigoplus_{A\in \cal{S}} \cal{H}(S,A)\otimes_{\rm{End}(A)} A^*,$$
where $ \cal{H}(A,S)\otimes_{\rm{End}(A)} A$ is defined by
$$ \cal{H}(A,S)\otimes_{\rm{End}(A)} A:= \int^{\{A\}} \cal{H}(A,S)\bp A,$$
where we take the coend over the full subcategory with single object $A$. 
\\ $\cal{H}(S,A) \otimes_{\rm{End}(A)} A^*$ is defined similarly.

\end{cor}

This in particular gives a correct version of  theorem 3.3.(a) in \cite{G}. One should replace $\cal{A}(\bb{Q}(0),B)\boxtimes B^\vee$  with $\cal{A}(\bb{Q}(0),B)\boxtimes_{\rm{End}(B)} B^\vee$ in loc.cit. and the result is the dual of $\pP$-module endomorphims of $\Psi$, not all of endomorphisms.

\subsection{Exact faithful module functors}\label{exact-faithful-functors}

In the following  theorem we present some reasonable conditions on $\pP$-module functors $\omega$ in Proposition \ref{ass-coalg}, such that the functor $\tilde{\omega}$ in the proposition, becomes equivalence:
\begin{theorem}\label{fiber}
Let $\cal{M}$ be an { abelian} locally finite module category over $\cal{P}$ satisfying the Chevalley condition. Let $S$ be an algebra in $\pP$ and $\omega: \mM\to \rm{Mod}(S)$ be an exact faithful $\pP$-module functor and for every $M\in \mM$, $\omega(M)$ is dualizable. Then the functor $\tilde{\omega}$ in Proposition \ref{ass-coalg} is an equivalence of $\pP$-module categories between $\mM$ and $\rm{coMod}(L(\omega))$. 
\end{theorem}

The main ingredient of the proof is a consequence of the Barr-Beck theorem for abelian categories, which is stated in \cite{D}, 4.1:
\begin{quotation}
Let $(T,U)$ be a pair of adjoint functors $T:\cal{A}\to \cal{B}$ and $U:\cal{B}\to \cal{A}$ between abelian categories $\cal{A}$ and $\cal{B}$. Then if $T$ is exact and faithful, $T$ induces an equivalence between $\cal{A}$ and the category of objects of $\cal{B}$ equipped with a coaction of the comonad $TU$.
\end{quotation}

\begin{proof} First we assume $\mM$ is finitely generated over $\pP$. Then By Theorem \ref{gGabber}, we can assume $\mM=\rm{Mod}(R)$ for a finite algebra $R$ in $\pP$. Let $M= \omega(R)$, then by Lemma \ref{right-exact}, $M$ is a $(R,S)$-bimodule and we can assume $\omega(-) = -\otimes_R M$.

Now $M$ is dualizable over $S$ and so there is a natural isomorphism: 
$$\Hom_S(X\otimes_R M, Y) \simeq \Hom_R(X, Y\otimes_S M^*)$$
for every $X\in \Mod(R)$. So $\omega$ has a right adjoint (denoted by $\alpha$), but $\omega$ is also exact and faithful, so by a consequence of the Barr-Beck theorem, $\omega$ induces an equivalence between $\mM$ and right $S$-modules with a coaction of the comonad $\omega\circ \alpha$. 

Now if we define $L= M^*\otimes_R M$, then $\omega\circ\alpha(-) = -\otimes_S L $ and the  comultiplication and counit morphisms of $\omega\circ\alpha$ induce corresponding morphisms on $L$ which satisfy the axioms for a coalgebroid over $S$. These morphisms can be explicitly written in terms of $i_M$ and $e_M$:
$$L\stackrel{\Delta}{\longrightarrow} L\otimes_S L = M^*\otimes_R M \xrightarrow{id \otimes i^R_M \otimes id} M^*\otimes_R M\otimes_S M^*\otimes_R M,$$
$$L\stackrel{\epsilon}{\longrightarrow} S = M^*\otimes_R M \stackrel{e^R_M}{\longrightarrow} S.$$
Also a coalgebra over $\omega\circ\alpha$ corresponds to a comodule over $L$.

We claim $L\simeq L(\omega)$. Let $L'$ be an arbitrary $(S,S)$-bimodule, then by Lemma \ref{right-exact}:
\begin{eqnarray*}
\Hom_\pP(\omega, \omega\otimes_S L') &\simeq & \Hom_\pP(-\otimes_R M, -\otimes_R(M\otimes_S L'))\\
& \simeq & \Hom_{(R,S)}(M, M\otimes_S L') \\
&\simeq & \Hom_{(S,S)}(M^*\otimes_R M, L') = \Hom_{(S,S)}(L,L'). 
\end{eqnarray*}
And the universal property of $L(\omega)$ implies the claim. Therefore for finitely generated $\mM$, $\omega$ induces an equivalence between $\mM$ and comodules over $L(\omega)$.

Now consider the general case of not finitely generated $\mM$. For an object $X\in \mM$, let us denote the full subcategory of subquotients of $A\otimes X, (A\in \pP)$ by $\langle X\rangle$. Then $\langle X\rangle$'s constitute a class of finitely generated $\pP$-subcategories of $\mM$ and $\mM$ is the directed union of these subcategories. We claim $L(\omega)$ is also a directed union of $L(\omega|\langle X\rangle)$'s. 
The only thing should be checked is that for two $X_1,X_2\in \mM$ with $\langle X_1\rangle \subset \langle X_2\rangle$ the induced morphism $L(\omega|\langle X_1\rangle) \to L(\omega |\langle X_2\rangle)$ is injective. The proof is very similar to (\cite{D}, 6.1).

By Lemma \ref{submodule}, the inclusion $\langle X_1\rangle\subset \langle X_2\rangle$ is equivalent to an inclusion $\rm{Mod}(R/I)\subset \rm{Mod}(R)$. If $\omega|\langle X_2\rangle$ be isomorphic to $-\otimes_R M$, we have $\omega|\langle X_1\rangle \simeq -\otimes_{R/I} (R/I\otimes_R M)$. Therefore
\begin{eqnarray*}
L(\omega|\langle X_1\rangle)&\simeq & (R/I\otimes_R M)^* \otimes_{R/I} (R/I\otimes_R M) \\ 
&\simeq & (R/I \otimes_R M)^* \otimes_R M \hookrightarrow M^*\otimes_R M \simeq L(\omega|\langle X_2\rangle).
\end{eqnarray*}
($R/I\otimes_R M$ is projective over $S$ and so it is a direct summand of $M$) and the composition of above morphisms is the natural morphism $L(\omega|\langle X_1\rangle) \to L(\omega |\langle X_2\rangle)$.

So $L(\omega)$ is the directed union of $L(\omega|\langle X\rangle)$'s and $\rm{coMod}(L(\omega))$ is also the union of $\rm{coMod}(L(\omega|\langle X\rangle))$'s. This fact completes the proof that $\tilde{\omega}$ is an equivalence on the whole of $\mM$.
\end{proof}

\section{Monoidal module categories}\label{mon-mod}

\subsection{Tensor product of module categories}
In this subsection we construct a notion of tensor product of module categories which will be used in the proof of the main Theorem \ref{main}.

\begin{defn}
Let $(\mM,m,r)$ be a right $\pP$-module category and $(\nN,n,l)$ a left $\pP$ category and $\cC$ an arbitrary ($k$-linear) category. By a $\pP$-balanced functor $(F,b):\mM\otimes \nN\to \cC$, we mean a functor $F: \mM\otimes \nN\to \cC$ together with natural isomorphisms $b_{M,P,N}:F((M\otimes P),N) \to F(M, (P\otimes N))$ for every $M\in \mM, P\in \pP, N\in \nN$, such that the  diagrams:

$$\xymatrix{F(M\otimes(P\otimes Q), N) \ar[rr]^{b_{M,(P\otimes Q),N}} \ar[d]_{m_{M,P,Q}} && F(M,(P\otimes Q)\otimes N) \ar[d]^{n_{P,Q,N}} \\
F((M\otimes P)\otimes Q,N)\ar[dr]_{b_{M\otimes P, Q,N}} && F(M,P\otimes (Q\otimes N)) \\
&F(M\otimes P,Q\otimes N)\ar[ur]_{b_{M,P,Q\otimes N}}&}$$
and
$$\xymatrix{F(M\otimes {\bf 1}, N)\ar[rd]_{r_M}\ar[rr]^{b_{M,{\bf 1},N}} && F(M,{\bf 1}\otimes N) \ar[dl]^{l_N}\\
&F(M, N)&}$$
commute.
\end{defn}

Now for $\mM,\nN,\pP$ as in the definition, $\mM\otimes_\pP \nN$ is a ($k$-linear) category together with a universal $\pP$-balanced functor: $(G,g): \mM\otimes \nN\to \mM\otimes_\pP\nN$. Which by universality of $(G,g)$ we mean that for every $\pP$-balanced functor $(F,b):\mM\otimes\nN\to \cC$, there is a unique functor $\tilde{F}: \mM\otimes_\pP\nN\to \cC$ such that $F= \tilde{F}\circ G$ and $b = \tilde{F}(g)$.

Such a universal category always exists and can be constructed as follows: $\mM\otimes_\pP \nN$ is the $k$-linear category obtained from $\mM\otimes \nN$ by adding new invertible morphisms 
$$g_{M,P,N}: (M\otimes P, N) \to (M,P\otimes N),$$
for every $M\in \mM, N\in \nN, P\in \pP$ and relations corresponding to commutativity of the  diagrams: ($f_1:M\to M', f_2:P\to P', f_3: N\to N'$ are morphisms)

$$\xymatrix{(M\otimes P,N) \ar[rr]^{g_{M,P,N}}\ar[d]_{(f_1\otimes f_2, f_3)} && (M,P\otimes N)\ar[d]^{(f_1,f_2\otimes f_3)} \\
(M'\otimes P', N') \ar[rr]^{g_{M',P',N'}}&& (M',P'\otimes N') }$$
and
$$\xymatrix{(M\otimes(P\otimes Q), N) \ar[rr]^{g_{M,(P\otimes Q),N}} \ar[d]_{m_{M,P,Q}} && (M,(P\otimes Q)\otimes N) \ar[d]^{n_{P,Q,N}} \\
((M\otimes P)\otimes Q,N) \ar[dr]_{g_{M\otimes P, Q,N}} && (M,P\otimes (Q\otimes N)) \\
&(M\otimes P, Q\otimes N)\ar[ur]_{g_{M,P,Q\otimes N}}&}$$
and
$$\xymatrix{(M\otimes {\bf 1},N) \ar[rd]_{r_M}\ar[rr]^{g_{M,{\bf 1},N}} && (M,{\bf 1}\otimes N )\ar[dl]^{l_N}\\
& (M,N) &}$$

Now it is evident that the {extension} functor $G:\mM\otimes \nN\to \mM\otimes_\pP \nN$ together with natural isomorphism $g$ is a universal $\pP$-balanced functor. 

When $\mM$ is a $(\qQ,\pP)$-bimodule category with middle associativity constraints 
$$d_{Q,M,P}: (Q\otimes M)\otimes P\xrightarrow{\simeq} Q\otimes (M\otimes P),$$
(for the definition of bimodule categories, cf. \cite{E}, 7.1), then $\mM\otimes_\pP \nN$ is a left $\qQ$-module category. This structure is induced by the left $\qQ$-module structure of $\mM\otimes\nN$ and the action of $\qQ$ on $g_{M,P,N}$'s which is given by: ($Q\in \qQ$)
$$Q\otimes g_{M,P,N} : (Q\otimes(M\otimes P), N)\rightarrow (Q\otimes M, P\otimes N) := g_{Q\otimes M, P,N} \circ (d_{Q,M,P}^{-1},1).$$
A similar structure holds when $\nN$ is a bimodule category.

Especially in the case where $\pP$ is symmetric, every left module category over $\pP$ can also be considered as a right module category and therefore a $(\pP,\pP)$-bimodule category. So for every two left $\pP$-module categories $\mM$ and $\nN$, $\mM\otimes_\pP \nN$ has a well defined structure of a left $\pP$-module category induced by the left action of $\pP$ on $\mM$, and a right $\pP$-module structure induced by the right action of $\pP$ on $\nN$, and these two $\pP$-module structures are naturally isomorphic.

\begin{rem}
In \cite{Gre}, J. Greenough introduces a notion of ``tensor product" of $k$-linear abelian module categories over $\pP$, denoted by $\boxtimes_\pP$, which is a generalization of Deligne's construction in \cite{D} of tensor products of abelian categories. But here we only need a notion of tensor product which is a $k$-linear additive category and so our construction is simpler.
\end{rem}

\subsection{The main theorem}
In this section we assume further that $\pP$ is symmetric. 

For every $(S,S)$-bimodule $L$ over a commutative algebra $S$, we denote by $L^{op}$ the $(S,S)$-bimodule which its underlying object is equal to $L$, but the left action of $S$ on it is given by the right action of $S$ on $L$ and the right action is given by the left action of $S$ on $L$. In the case where $L$ is a coalgebroid over $S$, $L^{op}$ has also a structure of a coalgebroid with the same counit and comultiplication defined by the composition of the comultiplication of $L$ with a switch $L\otimes_S L\to L^{op}\otimes_S L^{op}$.

If we denote the product and unit morphisms of $S$ by $\tilde{m}$ and $\tilde{u}$, then there is a natural coalgebroid structure on $S\otimes S$ over $S$, with its evident bimodule structure over $S$. The counit is given by $\tilde{m}$ and the coproduct is given by:
$$id \otimes \tilde{u} \otimes id: S\otimes S\to S\otimes S\otimes S = (S\otimes S)\otimes_S (S\otimes S).$$
\begin{defn}
Let $(S,\tilde{m},\tilde{u})$ be a commutative algebra in $\pP$, a commutative \emph{Hopf algebroid} over $S$, is a coalgebroid $(L,\Delta,\epsilon)$ over $S$ equipped with unit and multiplication morphisms of coalgebroids over $S$:
$$u: S\otimes S\to L, \quad m: L\otimes_{S\otimes S} L\to L,$$
such that $m$ is associative and commutative and $u$ is a unit for $m$. Furthermore there is a coalgebroid morphism
$$T: L^{op}\to L^{},$$
called antipode with the following property:
\begin{multline*}
L \xrightarrow{\Delta}L\otimes_S L \xrightarrow{}L^{op}\otimes_{S\otimes S} L \xrightarrow{T \otimes id} L\otimes_{S\otimes S} L \xrightarrow{m} L = \\ 
L \xrightarrow{\epsilon} S \xrightarrow{\tilde{u}\otimes - } S\otimes S \xrightarrow{u} L. 
\end{multline*}
In the special case $S=\bf 1$, a Hopf algebroid is called a \emph{Hopf algebra}.
\end{defn}

As in the case of commutative Hopf algebras over fields, the category of comodules over a Hopf algebroid $L$ over $S$ which are dualizable $S$-modules, is a rigid symmetric monoidal category, which is not necessarily abelian.

Our main theorem that generalizes classical the fundamental theorem of classical Tannakian categories (cf. \cite{D}, 1.12), is the following:
\begin{theorem}\label{main}
Let $\mM$ be an abelian $k$-linear rigid symmetric monoidal category, and $s:\pP\to \mM$ a $k$-linear symmetric monoidal functor. Then $s$ defines a $\pP$-module structure on $\mM$. Assume that $\mM$ is {locally finite over $\pP$} and satisfies the Chevalley property as a module category over $\pP$. 
Let $S$ be a commutative algebra in $\pP$ and $i:\pP\to \Mod(S)$ is defined by $i(-) = -\otimes S$. Suppose $\omega:\mM\to \rm{Mod}(S)$ be a $k$-linear exact faithful and symmetric monoidal functor such that there is a  symmetric monoidal isomorphism $\nu:\omega\circ s \to {i}$ (which implies that $\omega$ is a $\pP$-module functor). Then $L(\omega)$ has a structure of a Hopf algebroid over $S$ and $\omega$ induces a $\pP$-module symmetric monoidal equivalence between $\mM$ and $\rm{coMod}(L(\omega))$.
\end{theorem}

The opposite category of every $\pP$-module category $\mM$ is also a $\pP$-module category with the module structure $(P,M)\mapsto P^* \otimes M$. 
\begin{lem}
Let $\mM$ be a module category over $\pP$, $S$ a commutative algebra in $\pP$ and $\omega$ a $\pP$-module functor from $\mM$ to the category of dualizable $S$-modules. 
\begin{itemize}
\item[(i)] Let $i:\pP\to \Mod(S)$ be defined by $i(-) = -\otimes S$, then $L(i)\simeq S\otimes S$.
\item[(ii)] Define $\omega^*: \mM^{op}\to \Mod(S)$ by $\omega^*(M) = \omega(M)^*$. Then $ L(\omega^*)\simeq L(\omega)^{op}$. 
\item[(iii)] By the universal property of $\mM\otimes_\pP\mM$, composition of $\omega\otimes\omega: \mM\otimes\mM\to \Mod(S)\otimes \Mod(S)$ with tensor product over $S$ induces a $\pP$-module functor:
$$\omega\otimes_S \omega:\mM\otimes_\pP\mM\to \Mod(S), \quad (X,Y)\mapsto \omega(X)\otimes_S \omega(Y),$$ 
and $L(\omega\otimes_S\omega) \simeq L(\omega)\otimes_{S\otimes S} L(\omega)$.
\end{itemize}
\end{lem}

\begin{proof}
(i) By the definition of associated coalgebroid and Yoneda's lemma, we have:
\begin{eqnarray*}
L(i) &= & \int^{\pP} \Hom_S(P\otimes S,S)\bp (P\otimes S) \\
&\simeq & \left(\int^{\pP} \Hom(P,S)\bp P\right)\otimes S \quad\simeq S\otimes S.
\end{eqnarray*}
it is easy to see that this isomorphism is an equivalence of coalgebroids over $S$.

(ii) In the proof of Proposition \ref{ass-coalg}, we showed that $L(\omega)$ is naturally equivalent to the coequalizer of two morphisms $\phi$ and $\psi$ in (\ref{eqn1}). Writing the similar diagram for $\omega^*$, we have

\begin{multline*}
L(\omega^*) \simeq \mathrm{coeq}\left(
\int^{P\in\pP} \int^{M\in\mM} \omega(M)\otimes P^* \otimes P \otimes \omega(M)^*  \xymatrix{\ar@<.4ex>[r]^(.5){\,\phi'\,}\ar@<-.4ex>[r]_(.5){\,\psi'\,} & }\right.\\
\left. \int^{M\in \mM} \omega(M)\otimes \omega(M)^*
\right),
\end{multline*}
with similar definitions for $\phi'$ and $\psi'$. Thus there is an evident isomorphism between $L(\omega^*)$ and $L(\omega)$, but this isomorphism changes the directions of the actions of $S$ and is an isomorphism to $L(\omega)^{op}$ as a coalgebroid over $S$.

(iii) Let $\alpha:\mM\otimes \mM\to \Mod(S)$ be the functor defined by $\alpha(M,N) = \omega(M)\otimes_S \omega(N)$. Then by considering the action of $\pP$ on the first factor of $\mM\otimes \mM$, $\alpha$ is a $\pP$-module functor and for the structure functor $G:\mM\otimes\mM\to\mM\otimes_\pP\mM$, we have $\alpha = (\omega\otimes_S\omega)\circ G$.

Now again by using the expression (\ref{eqn1}) for associated coalgebroid and right exactness of tensor product over $S$, we see that $L(\alpha)$ is the quotient of 
\begin{multline*}
L:= \int^{(M,N)\in\mM\otimes\mM} \omega(N)^*\otimes_S\omega(M)^*\otimes\omega(M)\otimes_S\omega(N)\simeq \\
\left(\int^\mM \omega(M)^*\otimes \omega(M)\right)\otimes_{S\otimes S}\left(\int^\mM \omega(M)^*\otimes \omega(M)\right)
\end{multline*}
obtaind by equalizing two morphisms:
\begin{multline}\label{eqn2}
\omega(N)^*\otimes_S\omega(M)^*\otimes P^*\otimes P \otimes \omega(M)\otimes_S\omega(N)\xrightarrow{id \otimes ev\otimes id } \\ \omega(N)^*\otimes_S\omega(M)^* \otimes \omega(M)\otimes_S\omega(N)\to L,
\end{multline}
and 
\begin{multline}\label{eqn3}
\omega(N)^*\otimes_S\omega(M)^*\otimes P^*\otimes P \otimes \omega(M)\otimes_S\omega(N)\xrightarrow{\simeq } \\ \omega(N)^*\otimes_S\omega(P\otimes M)^* \otimes \omega(P\otimes M)\otimes_S\omega(N)\to L,
\end{multline}
for every $M,N\in \mM$ and $P\in \pP$. For the computation of $L(\omega\otimes_S\omega)$ we have the same construction but we only need to compute the above coend over $\mM\otimes_\pP\mM$ which is obtained by adding isomorphisms $g_{M,P,N}$ to $\mM\otimes \mM$. So $L(\omega\otimes_S\omega)$ is the quotient of $L(\alpha)$ obtained by equalizing two morphisms
$$
\omega(N)^*\otimes_S\omega(P\otimes M)^* \otimes \omega(P\otimes M)\otimes_S\omega(N) \to L(\alpha),
$$
and
\begin{multline*}
\omega(N)^*\otimes_S\omega(P\otimes M)^* \otimes \omega(P\otimes M)\otimes_S\omega(N) \xrightarrow{\simeq} \\
\omega(P\otimes N)^*\otimes_S\omega( M)^* \otimes \omega(M)\otimes_S\omega(P\otimes N) \to L(\alpha),
\end{multline*}
for every $M,N,P$.

So $L(\omega\otimes_S\omega)$ can be viewed as the coequalizer of morphisms (\ref{eqn2}) and (\ref{eqn3}) and 
\begin{multline}\label{eqn4}
\omega(N)^*\otimes_S\omega(M)^*\otimes P^*\otimes P \otimes \omega(M)\otimes_S\omega(N)\xrightarrow{\simeq } \\ \omega(P\otimes N)^*\otimes_S\omega( M)^* \otimes \omega( M)\otimes_S\omega(P\otimes N)\to L,
\end{multline}
and therefore by right exactness of $\otimes_S$,
$$L(\omega\otimes_S\omega) \simeq L(\omega) \otimes_{S\otimes S}L(\omega),$$
which is also an isomorphism of coalgebroids.

\end{proof}

\begin{proof}[Proof of \ref{main}] First we note that if we define the action of $\pP$ on $\mM$ using the monoidal functor $s:\pP\to \mM$ and the tensor product of $\mM$, the monoidal isomorphism $\omega\circ s\simeq i$ shows that $\omega$ is a $\pP$-module functor. On the other hand, $\omega$ is monoidal and sends every object of $\mM$ (which is dualizable) to a dualizable object in $\Mod(S)$. So $\omega:\mM\to \Mod(S)$ satisfies all of conditions of Theorem \ref{fiber}, and if we write $L = L(\omega)$, we have $\mM\simeq \rm{coMod}(L)$. It remains to show that  extra structures on $\mM$ and $\omega$, induces a structure of Hopf algebroid on $L$ and $\mM$ is equivalent as a monoidal category with $\rm{coMod}(L)$.

For this note that by the conditions on $\omega$, we have the following diagrams which are commutative up to natural $\pP$-module equivalences:
$$
\xymatrix{ \mM\otimes_\pP \mM\ar[r]^{\otimes}\ar[d]_{\omega\otimes \omega} & \mM\ar[d]^{\omega} \\ 
\Mod(S)\otimes_\pP \Mod(S) \ar[r]^(.65){\otimes_S} & \Mod(S) } 
$$
and 
$$
\xymatrix{\mM^{op} \ar[r]^{(-)^*}\ar[dr]^{\omega^*}& \mM\ar[d]^\omega \\ &\Mod(S)} 
\quad \quad
\xymatrix{\pP\ar[r]^s\ar[rd]_i & \mM\ar[d]^\omega \\ & \Mod(S)}$$

By the first diagram we have an equivalence $\omega\otimes_S\omega \simeq \omega(-\otimes -)$. So the universal $\pP$-module morphism $\omega\to \omega\otimes_S L$, gives a $\pP$-module morphism $\omega\otimes_S\omega \to (\omega\otimes_S\omega)\otimes_S L$ and so there is a corresponding morphism of coalgebroids:
$$m: L(\omega\otimes_S\omega)\simeq L\otimes_{S\otimes S} L \to L.$$
Similarly the second and third diagrams induce corresponding morphisms:
$$T: L(\omega^*)\simeq L^{op} \to L,\qquad u:L(i)\simeq S\otimes S \to L.$$

These morphisms satisfy the required conditions for a Hopf algebroid because of the conditions on the monoidal structure of $\mM$ and $\omega$. The proofs of these facts are completely parallel to the proofs of  similar results in the characterization of Tannakian categories (cf. \cite{D}, \S 6).
\end{proof}

Trying to obtain a  converse statement for the above theorem,  leads naturally to the following question:
\begin{question}
Let $S$ be an algebra in $\pP$ and $L$ a commutative Hopf algebroid over $S$ and define $\omega:\coMod(L)\to \Mod(S)$ to be the forgetful functor. For which classes of Hopf algebroids $L$, $\coMod(L)$ is rigid (or equivalently, the image of $\omega$ lies in the subcategory of dualizable $S$-modules) and the  morphism $L(\omega)\to L$ (corresponding to the universal property of $L(\omega)$) is isomorphism?
\end{question}
In the classical case, $\pP=\rm{Vect}_k$, the answer of the above question is the class of all commutative Hopf algebroids $L$ over $S$, which are faithfully flat
over $S\otimes S$ (cf. \cite{D}, Theorem 1.12, (iii)). An easy argument shows that the same thing holds for a general neutral Tannakian category $\pP$, 
but we do not know the answer in general.

\begin{rem}
In the special case $S=\bf 1$, $\Mod(S) = \pP$ and so the set of morphisms between every two object in $\Mod(S)$ is finite dimensional over $k$ and all objects have finite length. So if $\omega:\mM\to \pP$ be a faithful $\pP$-module functor, $\mM$ has also this properties. Furthermore the inclusions
$$\Hom_\mM(P\otimes M, N) \hookrightarrow \Hom_\pP(P\otimes \omega(M),\omega(N)), \quad (M,N\in\mM, P\in \pP)$$
induces an inclusion 
$$\uHom_\mM(M,N) \hookrightarrow \uHom_\pP(\omega(M),\omega(N))\simeq \omega(M)^*\otimes\omega(N),$$
and so the internal hom's of objects in $\mM$ are in $\pP$ and $\mM$ is locally finite over $\pP$. So in the special case $S=\bf 1$, we can remove the locally finiteness condition of above theorem and $\mM$ will be equivalent to the category of comodules over a Hopf algebra in $\pP$ (or equivalently ``finite representations of an affine group scheme in $\pP$"). This result (with small modifications) can be obtained also by Deligne's notion of the ``fundamental group" of a tensor category, which we will discuss in the next section.
\end{rem}

\section{Fundamental group of tensor categories}\label{fund-grp}

Recall that a $k$-tensor category (\cite{D}, 2.1) is a $k$-linear abelian rigid symmetric monoidal category with further condition, $\mathrm{End}({\bf 1})  =k.$ A \emph{Tannakian category}  is a $k$-tensor category together with a fiber functor (i.e. $k$-linear exact symmetric monoidal functor) to the category of $S$-modules for a nonzero commutative $k$-algebra $S$.


For every two locally finite $k$-linear abelian categories $\mathcal{A}_1,\mathcal{A}_2$, there is a $k$-linear abelian category\footnote{We used this notation to distinguish it from our notation $\otimes$ for product of categories.} $\mathcal{A}_1\otimes^D \mathcal{A}_2$ with a bilinear functor $\otimes: \mathcal{A}_1\times \mathcal{A}_2\to \mathcal{A}_1\otimes^D \mathcal{A}_2$ which is right exact in each variable, and for every bilinear functor $F:\mathcal{A}_1\times \mathcal{A}_2\to \cC$ to a $k$-linear abelian category which is right exact in each variable, there exists a $k$-linear right exact functor $\tilde{F}: \mathcal{A}_1\otimes^D \mathcal{A}_2\to \cC$ such that $F = \tilde{F}\circ \otimes$ and $\tilde{F}$ is unique up to equivalence. (cf. \cite{D}, \S 5)

If $\cal{T}$ is a locally finite $k$-tensor category, then there is an induced monoidal structure on $\cal{T}\otimes^D\cal{T}$ and $\cal{T}\otimes^D\cal{T}$ is a $k$-tensor category if $k$ is perfect or $\cal{T}$ is Tannakian. (\cite{D}, 8.1)

If $\omega:\cal{A}\to \cal{T}$ is a functor between a $k$-linear category $\cal{A}$ and  a $k$-tensor category $\cal{T}$, define:
$$\Lambda(\omega) := \int^{A\in \cal{A}} \omega(A)^*\otimes \omega(A),$$
it is an object of $\rm{ind}(\cal{T})$ and there is a universal morphism $\omega\to\omega\otimes\Lambda(\omega) $ such that for every $\Lambda\in\rm{ind}(cal{T})$ induces an isomorphism: (\cite{D}, 8.4) 
$$\Hom(\omega,\omega\otimes \Lambda)\simeq \Hom(\Lambda(\omega),\Lambda).$$ 

In the case $\omega:\cal{T}_1\to \cal{T}_2$ be a exact tensor functor (i.e. symmetric monoidal) between $k$-tensor categories, $\Lambda(\omega)$ is equipped with a structure of a Hopf algebra in $\cal{T}_2$ and there is a corresponding ``affine group scheme" \footnote{For basic algebraic geometry notions in tensor categories, see \cite{D}, 7.8} in $\tT_2$ (or a $\cal{T}_2$-group):
$$\pi(\omega):= \rm{Sp}(\Lambda(\omega)).$$
By definition this group represents the functor 
$$\rm{Sp}(R)\mapsto \rm{Isom}^\otimes(\omega\otimes R,\omega\otimes R ), $$
for commutative algebras $R$ in $\cal{T}_2$. (\cite{D}, 8.12-8.15)

In the special case $\omega = id_\cal{T}$ for a $k$-tensor category $\cal{T}$, we denote $\pi(id_\cal{T})$ by $\pi(\cal{T})$ and call it the \emph{fundamental group} of $\cal{T}$. So every object of $\cal{T}$ has an action of $\pi(\cal{T})$ which is called the \emph{natural action}.\footnote{The fundamental group has also a natural action on every algebra in the category and because the category of affine schemes is the opposite category of the category of algebras, it has also a right action on every affine scheme. But we use the word ``natural action" in the case of affine schemes for the inverse of this action which is also an action from left.} (\cite{D}, 8.12, 8.13) The natural action of $\pi(T)$ on itself is equal to the conjugation action, this is a special case of Lemma \ref{nat-act} below. 

For a general exact tensor functor $\omega:\cal{T}_1\to \cal{T}_2$ between $k$-tensor categories, the exactness and monoidal properties of $\omega$ immediately shows that
$$\pi(\omega) \simeq \omega(\pi(\cal{T}_1)),$$
and the universal property of $\Lambda(\omega)$ gives a morphism:
$$\Lambda(\omega) = \Lambda(id_{\cal{T}_2}\circ \omega)\to \Lambda(id_{\cal{T}_2}) ,$$
which induces a morphism of $\cal{T}_2$-groups: (\cite{D}, 8.15)
\begin{equation}\label{nu}
\nu:\pi(\cal{T}_2)\to \omega(\pi(\cal{T}_1)).
\end{equation}

\begin{lem}\label{nat-act}
With above notations, the natural action of $\pi(\cal{T}_2)$ on $\omega(\pi(\cal{T}_1))$ (as an affine scheme in $\cal{T}_2$) is equal to the conjugation action of $\pi(\cal{T}_2)$ by $\nu$ on $\omega(\pi(\cal{T}_1))$.
\end{lem}
\begin{proof}
The application of $\omega$ to the natural action of $\pi(\tT_1)$ on each $X\in \tT_1$, gives an action of $\omega(\pi(\tT_1))$ on every $\omega(X)$ (denoted by $\rho_X$) and by definition of $\nu$ the natural action of $\pi(\tT_2)$ on $\omega(X)$ is the composition of  this action with $\nu$ (denoted by $\sigma_X$). We also denote the natural action of $\pi(\tT_2)$ on $\omega(\pi(\tT_2))$ by $\sigma$.

Let $R$ be an algebra in $\tT_2$ and $X\in \tT_1$. Now for arbitrary $g\in \pi(\tT_2)(R)$, the following diagram is commutative: (it is because all morphisms in $\tT_2$ respect the natural action)
$$\xymatrix{
\omega(\pi(\tT_1))(R)\times (\omega(X)\otimes R) \ar[rr]^(.65){\rho_X}\ar[d]_{\sigma(g)\times \sigma_X(g)} &&\omega(X)\otimes R\ar[d]^{\sigma_X(g)}\\
\omega(\pi(\tT_1))(R)\times (\omega(X)\otimes R) \ar[rr]^(.65){\rho_X} &&\omega(X)\otimes R}
$$

So for every $h\in \omega(\pi(\tT_1))(R)$, we have
\begin{align*}
\rho_X(\sigma(g)(h)) &= \sigma_X(g)\circ \rho_X(h)\circ \sigma_X(g^{-1}) \\
&= \rho_X(\nu(g)) \circ \rho_X(h)\circ \rho_X(\nu(g^{-1})) = \rho_X( \nu(g)h\,\nu(g)^{-1}).
\end{align*}

Now 
$$\omega(\pi(\tT_1))(R) \simeq \pi(\omega)(R) \simeq \rm{Isom}^\otimes(\omega\otimes R, \omega\otimes R),$$
so every element of $\omega(\pi(\tT_1))(R)$ is determined by its action on $\omega(X)\otimes R$'s. So above equation shows that $\sigma(g)(h) = \nu(g)h\nu(g)^{-1}$ and the proof is complete. 
\end{proof}

\begin{theorem}[\cite{D}, 8.17] \label{deligne} Suppose $\cal{T}_1$ and $\cal{T}_2$ be two locally finite $k$-tensor categories and $\omega$ an exact tensor functor from $\cal{T}_1$ to $\cal{T}_2$. Suppose also that $\cal{T}_2\otimes^D\cal{T}_2$ is a tensor category over $k$. Then the tensor functor from $\cal{T}_1$ to the category of objects of $\cal{T}_2$ with an action of $\omega(\pi(\cal{T}_1))$ such that the composition of this action with the  morphism $\nu:\pi(\cal{T}_2)\to \omega(\pi(\cal{T}_1))$ (in (\ref{nu})) is the natural action, is equivalence. 
\end{theorem}
We use the above theorem to prove the following corollary, which is a variant  to the special case $S=\bf 1$ of \ref{main}.
\begin{cor}\label{fiber-tensor}
Suppose $\cal{T}_1$ and $\cal{T}_2$ be two locally finite $k$-tensor categories and $\omega$ an exact tensor functor from $\cal{T}_1$ to $\cal{T}_2$. Suppose also that $\cal{T}_2\otimes^D\cal{T}_2$ is a tensor category over $k$. Then if there exists an exact tensor functor $s:\cal{T}_2\to \cal{T}_1$ with tensor isomorphism $\alpha: \omega\circ s \to id_{\cal{T}_2}$, then there is a { $\tT_2$-group $G$ such that $\cal{T}_1$ is tensor equivalent to $\rm{Rep}_{\tT_2}(G)$ and $\omega$ is equal to the forgetful functor $\rm{Rep}_{\tT_2}(G)\to \cal{T}_2$ under this equivalence.} 
\end{cor}
Note that here we do not assume $\tT_2$ is semisimple.
\begin{proof}
Let $\nu_1:\pi(\tT_2)\to \omega(\pi(\tT_1))$ and $\nu_2: \pi(\tT_1) \to s(\pi(\tT_2))$ be the morphisms induced by $\omega$ and $s$. Then the morphism $\pi(\tT_2)\to \omega\circ s(\pi(\tT_2))$ corresponding to $\omega\circ s$ is equal to the composition $\omega(\nu_2)\circ \nu_1$ and is an isomorphism because of the commutativity of the diagram:
$$
\xymatrix{\pi(\tT_2)\ar[rr]^(.4){\omega(\nu_2)\circ \nu_1}\ar[drr]_{id}&& \omega\circ s (\pi(\tT_2))\ar[d]^{\alpha}_\simeq\\ && \pi(\tT_2)}
$$ 
So if we let the $\tT_2$-group $G=\rm{Sp}(L)$ to be the kernel of $\omega(\nu_2)$,
then $\omega(\pi(\tT_1))\simeq \pi(\tT_2)\ltimes G$ and the conjugation action of $\pi(\tT_2)$ on $G$ is the restriction of its conjugation action on $\omega(\pi(\tT_1))$, which is the natural action. So this action is also the natural action. 
Now by Theorem \ref{deligne}, $\tT_1$ is equivalent to the category of objects of $ \tT_2$ together with an action $\rho$ of $\pi(\tT_2)\ltimes G$ such that the restriction of $\rho$ to $\pi(\tT_2)$ is the natural action. 

But an arbitrary action $\rho$ of $\pi(\tT_2)\ltimes G$ on an object is given by two actions $\rho_1$ and $\rho_2$ of $\pi(\tT_2)$ and $G$, with the property that for every $g\in \pi(\tT_2)(R), h\in G(R)$ ($R$ an algebra), $\rho_2(ghg^{-1}) = \rho_1(g)\rho_2(h)\rho_1(g)^{-1}$ or equivalently $\rho_2(ghg^{-1}) \rho_1(g)= \rho_1(g)\rho_2(h).$ 
But in the case $\rho_1$ is the natural action, this condition is always satisfied, because the conjugation action of $\pi(\tT_2)$ on $G$ is equal to the natural action and $\rho_2$ respects the natural action. This shows that $\tT_2$ is equivalent to the category of representations of $G$ and $\omega$ corresponds under this equivalence to the forgetful functor $\rm{coMod}(L) = \rm{Rep}_{\tT_2}(G) \to \tT_2.$
\end{proof}

\begin{rem}
It is easy to see that for every commutative algebra $R$ in $\tT_2$, $G(R)$ is equal to the group of tensor isomorphisms of $\omega\otimes R$ which  trivially act on $\omega(s(\tT_2))$. Viewing $\tT_1$ as a $\tT_2$-module category by $s$, shows that $G(R)$ is also equal to the $\tT_2$-module endomorphisms of $\omega$ which respect the tensor structure of $\omega\otimes R$ (Every tensor endomorphism of a functor from a rigid category is an isomorophism). Thus $L=\mathcal{O}(G)$ (the ring of regular functions on $G$) represents the functor:
$$X\in \tT_2 \mapsto \Hom_{\tT_2}(\omega, \omega\otimes X),$$
and so $L$ is isomorphic to the associated coalgebra of $\omega$, defined in previous sections. So  the above corollary is very similar to the special case $S=\bf{1}$ of Theorem \ref{main}, and the deference is only on the conditions assumed on $\tT_1 = \mM$ and $\tT_2=\pP$.

In fact,  with the assumtions in the above corollary, $\omega$ is faithful (cf. \cite{D}, Corollary 2.10) and therefore $s$ is a faithful embedding. Also by exactness of $s$, $s$ has a right adjoint $t$, and so for every two objects $M_1,M_2$ in $\tT_1$ and every $X\in \tT_2$, we have:
$$\Hom_{\tT_1}(s(X)\otimes M_1,M_2)\simeq \Hom_{\tT_1}(s(X),  M_2\otimes M_1^*) \simeq \Hom_{\tT_2}(X,t( M_2\otimes M_1^*)),$$
and so $\tT_1$  is locally finite as a module category over $\tT_2$. So if we define $\mM:=\tT_1$, $\pP:=\tT_2$ and $S=\bf{1}$, we  only need two conditions for applying Theorem \ref{main}, the first is the semisimplicity of $\tT_1$ and the second is the Chevalley property. Again using Remark \ref{Chevalley}, we see that the second condition is also satisfied in the case of Tannakian categories in characteristic 0. So in this case, the above corollary for semisimple tensor categories $\tT_2$, can also be proved using Theorem \ref{main}.
\end{rem}
\begin{rem}\label{grp-sec}
 In this remark we show that the $\tT_2$-group $G$ defined in the proof of the above corollary depends on the section $s$ and it can not be determined only by $\omega$. For this consider a finite group $K$ with a subgroup $i:H \hookrightarrow K $, with two morphisms $s_1,s_2: K\to H$ such that $s_1 \circ i = s_2 \circ i = id_H$ and the  kernels $G_i = \mathrm{Ker}(s_i), \, i=1,2$ are not isomorphic. (We give an example below) 

let us define $\tT_1 = \mathrm{Rep}_k(K)$ and $\tT_2 = \mathrm{Rep}_k(H)$. Then the morphism $i:H\to K$ induces an exact tensor functor $\omega:\tT_1\to \tT_2$ and the induced functors $\tilde{s}_1, \tilde{s}_2$ corresponding to $s_1$ and $s_2$ give sections for $\omega$. Now if we forget the $\tT_2$ structures, $\pi(\tT_2)\simeq H$ and $\omega(\pi(\tT_1))\simeq K$ and so for $s=\tilde{s}_i, \, i=1,2$ the kernel of $\omega(\nu_2) =s_i $  is equal to $G_i$. So different sections $\tilde{s}_1, \tilde{s}_2$ give to nonisomorphic groups for $G$.

For an example of $K$ and $H$ satisfying above conditions consider
$$G =\langle x,y,z | x^3 = y^2 = z^2 = [x,y] = [y,z] = (xz)^2 =1 \rangle, \; H = \langle z\rangle. $$
We can define two  retractions $s_1$ and $s_2$ by
$$s_1(x)=s_2(x) = 1, \; s_1(y)=1, s_2(y)=z, \; s_3(z)=s_3(z) =z.$$
Then it can be easily verified that 
$$\mathrm{Ker}(s_1) = \langle x,y\rangle\simeq \mathbb{Z}/6\mathbb{Z}, \;\mathrm{Ker}(s_2) = \langle x, yz\rangle\simeq S_3,$$
and so $G_1\not\simeq G_2$.\footnote{The example is taken from Derek Holt's answer in  \url{http://math.stackexchange.com/a/1668652/202376}.}
\end{rem}

\end{document}